\documentclass{amsart}
\usepackage{amsmath}
\usepackage{amsthm}
\usepackage{amssymb}
\usepackage[curve]{xypic}
\usepackage{dsfont}
\usepackage{latexsym}
\usepackage{varioref}
\originalTeX
\usepackage{cite}
\usepackage{overpic}

\def\IZ{\mathds{Z}}
\def\IN{\mathds{N}}
\def\IR{\mathds{R}}

\def\Inv{\mathrm{Inv}}

\def\Obj{\mathrm{Obj}}

\def\eps{\varepsilon}

\newcommand{\undefined}{\Diamond}

\DeclareMathOperator{\interior}{int}

\DeclareMathOperator{\cl}{cl}

\DeclareMathOperator{\Def}{\mathcal{D}}

\DeclareMathOperator{\Con}{\mathcal{C}}

\DeclareMathOperator{\Hom}{\mathrm{H}}

\DeclareMathOperator{\CSS}{CSS}

\newcommand{\Hull}{\mathcal{H}}

\newcommand{\HTop}{\mathcal{HT}}
\newcommand{\gradMod}{\mathrm{gradMod}}

\newcommand{\rep}[1]{\left[#1\right]}

\newtheorem{theorem}{Theorem}[section]
\newtheorem{lemma}[theorem]{Lemma}

\newtheorem{corollary}[theorem]{Corollary}

\theoremstyle{definition}
\newtheorem{definition}[theorem]{Definition}

\theoremstyle{remark}
\newtheorem{remark}[theorem]{Remark}
\newtheorem{example}[theorem]{Example}

\begin{document}
	\title[Nonautonomous Conley Index and Attractor-Repeller decompositions]{
		Nonautonomous Conley Index Theory\\
		The Homology Index and Attractor-Repeller decompositions}
	
	\author[A. J\"anig]{Axel J\"anig}
	
	\address
	{\textsc{Axel J\"anig}\\
		Institut f\"ur Mathematik\\
		Universit\"at Rostock\\
		18051 Rostock, Germany}
	
	\email{axel.jaenig@uni-rostock.de}
	
	
	\begin{abstract}%
		In a previous work, the author established a nonautonomous Conley index
		based on the interplay between a nonautonomous evolution operator and
		its skew-product formulation. This index is refined
		to obtain a Conley index for families of nonautonomous evolution operators.
		Different variants such as a categorial index, a homotopy index and a homology
		index are obtained. Furthermore, attractor-repeller decompositions
		and conecting homomorphisms are introduced for the nonautonomous setting.
	\end{abstract}
	
	\maketitle
	 In \cite{article_naci}, the author defined a nonautonomous Conley index
	 relying on the interplay between an evolution operator\footnote{or process}
	 and a skew-product formulation. While isolation happens in the skew-product
	 formulation, the index pairs and thus the index live in another space and refers
	 directly to the nonautonomous evolution operator. 

	 An important technical detail of defining the index is the class of index pairs under consideration.
	 In \cite{article_naci}, index pairs are always obtained in the skew-product formulation.
	 In this paper, it will be proved that, roughly speaking, the same index can be defined
	 using a broader class of index pairs based on the evolution operator 
	 instead of the skew-product formulation.

	 Firstly, we will formulate and prove an inclusion property for index pairs.
	 A homotopy index, a categorial index and a homology Conley index will
	 be introduced and, using the previously introduced inclusion properties,
	 shown to be well-defined. Most of these concepts have evolved over decades
	 and are only adapted\footnote{Each genuinely nonautonomous definition in this paper also applies to the autonomous setting. Therefore, a comparison is possible. Minor
	 	differences between our definition and other variants (such as \cite{hib}, for instance) might occur.
	 	} to the nonautonomous setting.
	 
	 A powerful feature of Conley index theories is certainly its ability
	 to reflect attractor-repeller decompositions obtained from the skew-product formulation. 
	 Passing to homology, an attractor-repeller
	 decomposition gives rise to a long exact sequence \cite{conn_matrix, connmatrix} and
	 a so-called connecting homomorphism.
	 These sequences contain information on the connections between attractor and repeller.
	 
	 Usually this long exact sequence is obtained from so-called index triples. Using
	 an appropriate adaption of index triples, these algebraic sequences and
	 their connecting homomorphisms are shown to be available for
	 the nonautonomous index, too.
	 
	 Following a Preliminaries section, a notion of related index pairs is introduced
	 in Section 2.
	 Based on these results, a categorial index is defined in Section 3. Section 4
	 is devoted to attractor-repeller decompositions based on the notion of
	 a homology Conley index defined there as well.
	 
	 The reader who is interested in applications is referred to \cite{article_naci}.
	 Continuation properties of Morse-decompositions
	 and a uniformity property of the connecting homomorphism will be discussed
	 in subsequent papers.

	 \begin{section}{Preliminaries}
	 	For the convenience of the reader, we collect important definitions
	 	and terminology from other sources, mostly following
	 	the author's previous paper on the subject \cite{article_naci}.
	 	
	 	  	\begin{subsection}{Quotient spaces}
	 	  		\index{quotient space}
	 	  		\begin{definition}
	 	  			\label{df:140606-1652}
	 	  			Let $X$ be a topological space, and $A,B\subset X$. Denote
	 	  			\begin{equation*}
	 	  			A/B := A/R \cup \{A\cap B\},
	 	  			\end{equation*}
	 	  			where $A/R$ is the set of equivalence classes
	 	  			with respect to the relation $R$ on $A$ which is defined by $x R y$ iff $x=y$ or $x,y\in B$.
	 	  			
	 	  			We consider $A/B$ as a topological space
	 	  			endowed with the quotient topology with respect to the
	 	  			canonical projection $q:\; A\to A/B$, that is, a
	 	  			set $U\subset A/B$ is open if and only if
	 	  			\begin{equation*}
	 	  			q^{-1}(U) = \bigcup_{x\in U} x
	 	  			\end{equation*}
	 	  			is open in $A$.
	 	  		\end{definition}
	 	  		
	 	  		Recall that the quotient topology is the final topology
	 	  		with respect to the projection $q$. 
	 	  		
	 	  		\begin{remark}
	 	  			The above definition is compatible
	 	  			with the definition used in \cite{hib}
	 	  			or \cite{ryb}. The only difference
	 	  			occurs in the case $A\cap B=\emptyset$, where
	 	  			we add $\emptyset$, which is never an equivalence class,
	 	  			instead of an arbitrary point.
	 	  		\end{remark}
	 	  	\end{subsection}
	 	  	
	 	  	\begin{subsection}{Evolution operators and semiflows}
	 	  		Let $X$ be a metric space. Assuming that $\undefined\not\in X$, we
	 	  		introduce a symbol $\undefined$, which means "undefined". The intention is
	 	  		to avoid the distinction if an evolution operator is defined for a given argument
	 	  		or not. Define $\overline A := A\dot\cup\{\undefined\}$ whenever $A$ is a set with $\undefined\not\in A$.
	 	  		Note that $\overline{A}$ is merely a set, the notation does not contain any
	 	  		implicit assumption on the topology. 
	 	  		
	 	  		\index{evolution operator}
	 	  		\index{semiflow}
	 	  		
	 	  		\begin{definition}
	 	  			\label{df:131009-1504}
	 	  			\index{evolution operator}
	 	  			Let $\Delta:=\{(t,t_0)\in\IR^+\times\IR^+:\; t\geq t_0\}$.
	 	  			A mapping $\Phi:\; \Delta\times \overline{X}\to \overline{X}$
	 	  			is called an {\em evolution operator} if
	 	  			\begin{enumerate}
	 	  				\item $\Def(\Phi) := \{(t,t_0,x)\in\Delta\times X:\; \Phi(t,t_0,x)\neq\undefined\}$ is open
	 	  				in $\IR^+\times\IR^+\times X$;
	 	  				\item $\Phi$ is continuous on $\Def(\Phi)$;
	 	  				\item $\Phi(t_0,t_0, x) = x$ for all $(t_0,x)\in\IR^+\times X$;
	 	  				\item $\Phi(t_2 ,t_0,x) = \Phi(t_2,t_1,\Phi(t_1,t_0,x))$ for all $t_0\leq t_1\leq t_2$ in $\IR^+$ and $x\in X$;
	 	  				\item $\Phi(t,t_0,\undefined) = \undefined$ for all $t\geq t_0$ in $\IR^+$.
	 	  			\end{enumerate}
	 	  			
	 	  			A mapping $\pi:\; \IR^+\times \overline{X}\to \overline{X}$ is called {\em semiflow}
	 	  			if $\tilde\Phi(t+t_0,t_0,x) := \pi(t,x)$ defines an evolution operator.
	 	  			To every evolution operator $\Phi$, there is an associated
	 	  			(skew-product) semiflow $\pi$ on an extended phase space $\IR^+\times X$, defined by $(t_0,x)\pi t = (t_0+t, \Phi(t+t_0,t_0,x))$.
	 	  			
	 	  			A function $u:\; I\to X$ defined on a subinterval $I$ of $\IR$ is called a {\em solution of $\Phi$}
	 	  			if $u(t_1) = \Phi(t_1,t_0,u(t_0))$ for all $\left[t_0,t_1\right]\subset I$.
	 	  		\end{definition}
	 	  		
	 	  		\begin{definition}
	 	  			\index{largest invariant subset}
	 	  			\index{largest positively invariant subset}
	 	  			\index{largest negatively invariant subset}
	 	  			\index{$\Inv$}
	 	  			\index{$\Inv^+$}
	 	  			\index{$\Inv^-$}
	 	  			
	 	  			Let $X$ be a metric space, $N\subset X$ and $\pi$ a semiflow on $X$. The set
	 	  			\begin{equation*}
	 	  			\Inv^-_\pi(N) := \{x\in N:\;\text{ there is a solution }u:\;\IR^-\to N\text{ with }u(0) = x\}
	 	  			\end{equation*}
	 	  			is called the {\em largest negatively invariant subset of $N$}.
	 	  			
	 	  			The set
	 	  			\begin{equation*}
	 	  			\Inv^+_\pi(N) := \{x\in N:\;x\pi\IR^+\subset N\}
	 	  			\end{equation*}
	 	  			is called the {\em largest positively invariant subset of $N$}.
	 	  			
	 	  			The set
	 	  			\begin{equation*}
	 	  			\Inv_\pi(N) := \{x\in N:\;\text{ there is a solution }u:\;\IR\to N\text{ with }u(0) = x\}
	 	  			\end{equation*}
	 	  			is called the {\em largest invariant subset of $N$}.
	 	  		\end{definition}

	 	  		\index{skew-product semiflow}
	 	  		\index{cocycle mapping}
	 	  		
	 	  		Let $X$ and $Y$ be metric spaces, 
	 	  		and assume that $y\mapsto y^t$ is
	 	  		a global\footnote{defined for all $t\in\IR^+$} semiflow 
	 	  		on $Y$, to which we will refer as $t$-translation.
	 	  		
	 	  		\begin{example}
	 	  			Let $Z$ be a metric space, and let $Y := C(\IR^+, Z)$
	 	  			be a metric space such that a sequence of functions converges
	 	  			if and only if it converges uniformly on bounded sets.
	 	  			The translation can now be defined canonically by
	 	  			$y^t(s) := y(t+s)$ for $s,t\in\IR^+$.
	 	  		\end{example}
	 	  		
	 	  		A suitable abstraction of many non-autonomous problems is given
	 	  		by the concept of skew-product semiflows introduced below.
	 	  		
	 	  		\begin{definition}
	 	  			We say that $\pi = (.^t,\Phi)$ is a skew-product semiflow on $Y\times X$
	 	  			if $\Phi:\; \IR^+\times \overline{Y\times X}\to \overline{Y\times X}$ is a mapping such that 
	 	  			\begin{equation*}
	 	  			(t,y,x)\pi t := 
	 	  			\begin{cases}
	 	  			(y^t, \Phi(t,y,x)) & \Phi(t,y,x)\neq\undefined\\
	 	  			\undefined & \text{otherwise}
	 	  			\end{cases}
	 	  			\end{equation*}
	 	  			is a semiflow on $Y\times X$.
	 	  		\end{definition}
	 	  		
	 	  		\begin{definition}
	 	  			\index{positive hull}
	 	  			\index{$\Hull^+(y_0)$}
	 	  			\index{$Y_c$}
	 	  			For $y\in Y$ let
	 	  			\begin{equation*}
	 	  			\Hull^+(y) := \cl_Y \{y^t:\;t\in\IR^+\} 
	 	  			\end{equation*}
	 	  			denote the positive hull of $y$. Let $Y_c$ denote the set of 
	 	  			all $y\in Y$ for which $\Hull^+(y)$ is compact.
	 	  		\end{definition}
	 	  		
	 	  		\begin{definition}
	 	  			\label{df:isolating-neighborhood-1}
	 	  			\index{isolating neighborhood in $\Hull^+(y_0)\times X$}
	 	  			Let $y_0\in Y$ and $N\subset \Hull^+(y_0)\times X$ be a closed subset.
	 	  			$N$ is called an {\em isolating neighborhood} (for $K$ in $\Hull^+(y_0)\times X$) if $\Inv N\subset \interior_{\Hull^+(y_0)\times X} N$ (and $K=\Inv N$).
	 	  		\end{definition}
	 	  		
	 	  		The following definition is a consequence of the slightly modified notion
	 	  		of a semiflow (Definition \ref{df:131009-1504}) but not a semantical change
	 	  		compared to \cite{hib}, for instance.
	 	  		
	 	  		\begin{definition}
	 	  			\index{does not explode}
	 	  			We say that $\pi$ explodes in $N\subset Y\times X$
	 	  			if $x\pi\left[0,t\right[\subset N$ and $x\pi t=\undefined$.
	 	  		\end{definition}
	 	  		
	 	  		Following \cite{ryb_homotopy_index}, we formulate the following
	 	  		asymptotic compactness condition.
	 	  		
	 	  		\begin{definition}
	 	  			A set $M\subset Y\times X$ is called {\em strongly admissible} provided
	 	  			the following holds:
	 	  			
	 	  			Whenever $(y_n,x_n)$ is a sequence in $M$ and $(t_n)_n$ is a sequence
	 	  			in $\IR^+$ such that $(y_n,x_n)\pi\left[0,t_n\right]\subset M$, then
	 	  			the sequence $(y_n,x_n)\pi t_n$ has a convergent subsequence.
	 	  		\end{definition}
	 	  		
	 	  		\begin{definition}
	 	  			Let $\pi = (.^t, \Phi)$ be a skew-product semiflow and $y\in Y$. Define
	 	  			\begin{equation*}
	 	  			\Phi_y(t+t_0,t_0,x) := \Phi(t,y^{t_0},x).
	 	  			\end{equation*}
	 	  			It is easily proved that $\Phi_y$
	 	  			is an evolution operator in the sense of Definition \ref{df:131009-1504}.
	 	  		\end{definition}
	 	  	\end{subsection}
	 \end{section}

	 \begin{section}{Related index pairs}
	 	In this section we give a definition of a nonautonomous
	 	Conley index which is slightly different from the index defined in \cite{article_naci}. 
	 	Essentially, the index is now purely based on nonautonomous
	 	index pairs which are subsets of $\IR^+\times X$, where $X$ is an appropriate metric space.
	 	It is often more convenient to compute the index by using the modified
	 	definition of this section.
	 	The main results are Theorem \ref{th:140128-1546} and its corollary.
	 	
	 	We say that two index pairs for which the assumptions and thus
	 	also the conclusions of Theorem \ref{th:140128-1546} hold
	 	are {\em related}. Roughly speaking, related index pairs
	 	define the same index\footnote{This is not necessarily a homotopy index, so the vague language is intended.}. 
	 	
	 	Throughout this section, it is assumed that $X$ and $Y$ are metric spaces,
	 	and $\pi = \pi(.^t,\Phi)$ is a skew-product semiflow on $Y\times X$. By $\chi := \chi_{y_0}$
	 	we denote the canonical semiflow $(t,x)\chi_{y_0}s := (t+s, \Phi_{y_0}(s,0,x))$
	 	on $\IR^+\times X$.
	 	
	 	\begin{definition}
	 		\label{df:basic_index_pair}
	 		A pair $(N_1, N_2)$ is called a {\em (basic) index pair} relative to a
	 		semiflow $\chi$ in $\IR+\times X$ if
	 		\begin{enumerate}
	 			\item[(IP1)] $N_2\subset N_1\subset \IR^+\times X$, $N_1$ and $N_2$ are
		 			closed in $\IR^+\times X$
			 	\item[(IP2)] If $x\in N_1$ and $x\chi t\not\in N_1$ for some $t\in\IR^+$, 
				 	then $x\chi s\in N_2$ for some $s\in\left[0,t\right]$;
				 \item[(IP3)] If $x\in N_2$ and $x\chi t\not\in N_2$ for some $t\in\IR^+$, then
					$x\chi s\in (\IR^+ \times X)\setminus N_1$ for some $s\in\left[0,t\right]$.
	 		\end{enumerate}
	 	\end{definition}
	 	
	 	The definition above establishes the core properties of an index pair
	 	and is taken from \cite{article_naci}. To obtain an index, we need to associate
	 	invariant sets with index pairs.

 	 	\begin{definition}
	 		\label{df:140128-1503}
	 		\index{index pair for $(y_0, K)$}
	 		Let $y_0\in Y$ and $(N_1, N_2)$ be a basic index pair in $\IR^+\times X$ relative to $\chi_{y_0}$. Define
	 		$r:=r_{y_0}:\; \IR^+\times X\to \Hull^+(y_0)\times X$ by $r_{y_0}(t,x) := (y^t_0,x)$.
	 		
	 		Let $K\subset \omega(y_0)\times X$ be an (isolated) invariant
	 		set. We say that $(N_1, N_2)$ is a (strongly admissible) index pair\footnote{Every index pair in the
	 			sense of Definition \ref{df:140128-1503} is assumed to be strongly admissible.} for $(y_0, K)$ if:
	 		\begin{enumerate}
	 			\item[(IP4)] there is a strongly admissible isolating neighborhood $N$ of $K$ in $\Hull^+(y_0)\times X$
	 			such that $N_1\setminus N_2 \subset r^{-1}(N)$;
	 			\item[(IP5)] there is a neighborhood $W$ of $K$ in $\Hull^+(y_0)\times X$
	 			such that $r^{-1}(W)\subset N_1\setminus N_2$.
	 		\end{enumerate}
	 	\end{definition}
	 	
	 	\begin{definition}
	 		\index{invariant pair}
	 		\index{compact invariant pair}
	 		We say that $(y_0, K)$ is an {\em invariant pair} if $y_0\in Y$ and $K\subset \Hull^+(y_0)\times X$.
	 		An invariant pair $(y_0, K)$ is called a {\em compact invariant pair} provided 
	 		that $K$ is compact.
	 	\end{definition}
	 	
	 	Every FM-index pair relative to the skew-product semiflow induces
	 	an index pair. Therefore, the homotopy index defined here and the homotopy index from
	 	\cite{article_naci} agree\footnote{A more detailed explanation can be found right after Theorem \ref{th:140128-1546}.}.
	 	
	 	\begin{lemma}
	 		\label{le:141113-1800}
	 		Let $y_0\in Y$ and let $(N_1, N_2)$ be an FM-index pair
	 		for $K\subset\Hull^+(y_0)\times X$ such that $N_1$ is
	 		strongly admissible. 
	 		
	 		Then
	 		$(M_1, M_2) := (r^{-1}_{y_0}(N_1), r^{-1}_{y_0}(N_2))$ is an index
	 		pair for $(y_0, K)$.
	 	\end{lemma}
	 	
	 	\begin{proof}
	 		$(M_1, M_2)$ is an index pair by Lemma 4.3 in \cite{article_naci}.
	 		We need to prove that the 
	 		assumptions (IP4) and (IP5) of Definition \ref{df:140128-1503} are satisfied.
	 		\begin{enumerate}
	 			\item[(IP4)] $N:=\cl_{Y\times X} (N_1\setminus N_2)$ is an isolating neighborhood
			  		for $K$, and $M_1\setminus M_2 = r^{-1}(N_1)\setminus r^{-1}(N_2) \subset r^{-1}(N)$.
		 		\item[(IP5)] Let $W:=\interior_{\Hull^+(y_0)\times X} (N_1\setminus N_2)$, which is
	 		a neighborhood of $K$. We have
	 		$r^{-1}(W) \subset r^{-1}(N_1)\setminus r^{-1}(N_2)$.
		 	\end{enumerate}
	 	\end{proof}
	 	
	 	The following lemma is not much more than a restatement of Theorem 3.5 in \cite{article_naci}.
	 	
	 	\begin{lemma}
	 		\label{le:140923-2008}
	 		Suppose that $(N_1, N_2)\subset (M_1, M_2)$ are index pairs for $(y_0,K)$. The
	 		inclusion induced mapping $i:\; (N_1/N_2, N_2) \to (M_1/M_2, M_2)$ is a homotopy
	 		equivalence.
	 	\end{lemma}
	 	
	 	\begin{proof}
	 		By Definition \ref{df:140128-1503},
	 		there is a neighborhood $W$ of $K$ such that $r^{-1}(W)\subset (N_1\setminus N_2)\cap (M_1\setminus M_2)$. 
	 		It follows from Definition \ref{df:140128-1503} that 
	 		the closure $\overline{W}:=\cl_{Y\times X}W$ is strongly admissible, so 
	 		by \cite[Lemma 4.3]{article_naci}, $(N_1, N_2)$ and $(M_1, M_2)$ are index pairs
	 		for $(\Phi_{y_0}, r^{-1}_{y_0}(W))$.
	 		The claim is now a direct consequence of Theorem 3.5 in \cite{article_naci}.
	 	\end{proof}
	 	
	 	\begin{definition}
	 		Let $(N_1, N_2)$ be an index pair in $\IR^+\times X$ (relative to the semiflow
	 		$\chi$ on $\IR^+\times X$). For $T\in\IR^+$, we set
	 		\begin{equation*}
	 		N^{-T}_2 := N^{-T}_2(N_1) := \{(t,x)\in N_1:\; \exists s\leq T\; (t,x)\chi s\in N_2\}.
	 		\end{equation*}
	 	\end{definition}
	 	
	 	\begin{lemma}
	 		\label{le:140924-1939}
	 		Let $(N_1, N_2)$ be an index pair for $(y_0,K)$. Then so is
	 		$(N_1, N^{-T}_2)$ for every $T\in\IR^+$.
	 	\end{lemma}
	 	
	 	\begin{proof}
	 		We need to check the assumptions of Definition \ref{df:basic_index_pair}
	 		and Definition \ref{df:140128-1503}.
	 		\begin{enumerate}
	 			\item[(IP1)]
	 			We need to show that $N^{-T}_2$ is closed. Suppose that $(s_n, x_n)$ is a sequence
	 			in $N^{-T}_2$ with $(s_n, x_n)\to (s,x)$ in $N_1$. For every $n\in\IN$, there is a
	 			$t_n\in\left[0,T\right]$ such that $(s_n,x_n)\chi t_n\in N_2$. We can assume 
	 			without loss of generality that $t_n\to t\leq T$, so $(s,x)\chi t\in N_2$, which is closed.
	 			Thus it holds that $(s,x)\in N^{-T}_2$.
	 			\item[(IP2)]
	 			Let $x\in N^{-T}_2$ but $x\chi t\not\in N_1$ for some $t\in\IR^+$. $(N_1, N_2)$ is an
	 			index pair, so $x\chi s\in N_2\subset N^{-T}_2$ for some $s\in\left[0,t\right]$. 
	 			\item[(IP3)]
	 			Suppose that $x\in N^{-T}_2$ and $x\chi t\not\in N^{-T}_2$ for some $t\in\IR^+$.
	 			Letting $t_0 := \sup\{s\in\IR^+:\; x\chi\left[0,s\right]\cap N_2 = \emptyset\}$, it follows
	 			that $t_0\leq T$ and $x\chi t_0\in N_2$. Furthermore, one has $x\chi\left[0,t_0\right]\subset N^{-T}_2$,
	 			so $t>t_0$. 
	 			
	 			Since $(N_1, N_2)$ is assumed to be an index pair, 
	 			it follows that $x\chi s\in (\IR^+\times X)\setminus N_1$
	 			for some $s\in\left[t_0, t\right]$.
	 		\item[(IP4)]
	 		$(N_1, N_2)$ is an index pair for $(y_0,K)$, so there is an isolating neighborhood
	 		$N$ of $K$ such that $N_1\setminus N^{-T}_2\subset N_1\setminus N_2\subset r^{-1}(N)$.

			\item[(IP5)]
	 		Let $W$ be an open neighborhood of $K$ such that $r^{-1}(W)\subset N_1\setminus N_2$. We 
	 		consider the set
	 		\begin{equation*}
	 		W^T := \{ (y,x)\in W:\; (y,x)\pi \left[0,T\right]\subset W \}.
	 		\end{equation*}
	 		If $(t,x)\in r^{-1}(W^T)\cap N^{-T}_2$, then $(t,x)\chi_{y_0} T\in r^{-1}(W)\cap N_2=\emptyset$,
	 		so 
	 		\begin{equation*}
	 		r^{-1}(W^T)\subset N_1\setminus N^{-T}_2. 
	 		\end{equation*}
	 		We need to show that $W^T$ is a neighborhood
	 		of $K$. Suppose to the contrary that there is\footnote{As a consequence of the admissibility assumption, $K$ is compact.} 
	 		a sequence $(y_n,x_n)\to (y'_0,x_0)\in K$ in
	 		$N\setminus W^T$. For every $n\in\IN$, there is a $t_n\in\left[0,T\right]$
	 		with $(y_n,x_n)\pi t_n\in (\Hull^+(y_0)\times X)\setminus W$.  We can assume w.l.o.g. that $t_n\to t_0$,
	 		so $(y'_0,x_0)\pi t_0\in (\Hull^+(y_0)\times X)\setminus W$, which is a closed set.
	 		However, $(y'_0,x_0)\pi t_0\in K\subset W$, a contradiction.
	 		\end{enumerate}
	 	\end{proof}
	 	
	 	One frequently needs to prove that a pair $(N_1, N_2)$ is not only an index
	 	pair but also that it belongs to a certain pair $(y_0, K)$. For this purpose
	 	and in conjunction
	 	with Lemma \ref{le:140924-1939}, the following -- simple -- "sandwich" lemma is useful.
	 	
	 	\begin{lemma}
	 		\label{le:150921-1746}
	 		Let $y_0\in Y$, and let
	 		$(N_1, N_2)$, $(M_1, M_2)$ and $(N'_1, N'_2)$ be index pairs with $N_1\setminus N_2\subset M_1\setminus M_2\subset N'_1\setminus N'_2$. 
	 		
	 		If $(N_1, N_2)$ and $(N'_1, N'_2)$ are index pairs for $(y_0, K)$, then so is $(M_1, M_2)$.
	 	\end{lemma}
	 	
	 	\begin{proof}
	 		One simply needs to check the assumptions of Definition \ref{df:140128-1503}.
	 		\begin{enumerate}
	 			\item[(IP4)]
	 			$(N'_1, N'_2)$ is an index pair for $(y_0, K)$, so there is a
	 			strongly admissible isolating neighborhood $N$ 
	 			of $K$ in $\Hull^+(y_0)\times X$ such that $M_1\setminus M_2\subset N'_1\setminus N'_2\subset r^{-1}(N)$.
	 			\item[(IP5)]
	 			$(N_1, N_2)$ is an index pair for $(y_0, K)$, so there is a
	 			neighborhood $W$ of $K$ in $\Hull^+(y_0)\times X$ such that 
	 			$r^{-1}(W)\subset N_1\setminus N_2\subset M_1\setminus M_2$.
	 		\end{enumerate}
	 	\end{proof}
	 	
	 	We are now in a position to formulate and prove the main result of this section.
	 	
	 	\begin{theorem}
	 		\label{th:140128-1546}
	 		Let there be given index pairs $(N_1, N_2)$ and $(M_1, M_2)$
	 		for $(y_0, K)$. Further, let $N\subset \Hull^+(y_0)\times X$ be a strongly admissible neighborhood of $K$.
	 		Then there are a $t_0\in\IR^+$ and an index pair $(L_1, L_2)$ 
	 		such that 
	 		\begin{equation*}
	 		(L_1, L_2) \subset (r^{-1}(N)\cap N_1\cap M_1, N^{-t_0}_2(N_1) \cap M^{-t_0}_2(M_1)).
	 		\end{equation*}
	 	\end{theorem}
	 	\index{homotopy index (nonautonomous, second variant)}
	 	
	 	An important consequence of the theorem above is that the homotopy
	 	index of $(y_0, K)$ can be defined as the pointed homotopy type 
	 	of $(N_1/N_2, N_2)$, where $(N_1, N_2)$ is an index pair for $(y_0, K)$.
	 	It coincides\footnote{Under the assumptions of Theorem \ref{th:140128-1546}, it follows from \cite{ryb} that
	 		there exists an isolating block for $K$ in $\Hull^+(y_0)\times X$. This isolating
	 		block gives rise to an index pair for $(y_0, K)$ as proved in Lemma \ref{le:141113-1800}.}
	 	with Definition 4.1 in \cite{article_naci}, so there is no need
	 	to redefine the homotopy index. We have merely extended the class of
	 	possible or good index pairs.
	 	
	 	\begin{corollary}
	 		\label{co:140123-1802}
	 		Under the assumptions of Theorem \ref{th:140128-1546},
	 		the pointed homotopy types of $(N_1/N_2, N_2)$ and $(M_1/M_2, M_2)$
	 		agree.
	 	\end{corollary}
	 	
	 	\begin{proof}
	 		By Theorem \ref{th:140128-1546}, there are an
	 		index pair and a constant $t_0\in\IR^+$ for which the following inclusions 
	 		hold true.
	 		\begin{align*}
	 		(L_1, L_2) \subset &(N_1, N^{-t_0}_2) \supset (N_1, N_2)\\
	 		(L_1, L_2) \subset &(M_1, M^{-t_0}_2) \supset (M_1, M_2)
	 		\end{align*}
	 		In view of Lemma \ref{le:140923-2008} and Lemma \ref{le:140924-1939}, 
	 		this readily implies that $(N_1/N_2, N_2)$ and $(M_1/M_2, M_2)$ are isomorphic
	 		in the homotopy category of pointed spaces.
	 	\end{proof}
	 	
	 	The rest of this section is devoted to the proof
	 	of Theorem \ref{th:140128-1546}. The proof is similar
	 	to the proof of \cite[Lemma 4.8]{hib}, but instead
	 	of using isolating blocks, we will construct appropriate 
	 	index pairs. In all subsequent lemmas, we will assume
	 	that the hypotheses of Theorem \ref{th:140128-1546} hold.
	 	
	 	Since $N$ is a neighborhood of $K$, there is an open (in $\Hull^+(y_0)\times X$)
	 	set $U$ with $K\subset U\subset N$. Define $g^+,g^-:\; \Hull^+(y_0)\times X\to \IR^+$
	 	by
	 	\begin{align*}
	 	g^+(y,x) &:= \sup\{ t\in\IR^+:\; (y,x)\pi\left[0,t\right]\subset U\}\\
	 	g^-(y,x) &:= \sup\{ d((y,x)\pi t,\Inv^-_\pi(N)):\; t\in\left[0,g^+(y,x)\right]\}.
	 	\end{align*}
	 	
	 	It is easy to see that both functions $g^+$ and $g^-$ are continuous and
	 	monotone decreasing along solutions in $U$ (resp. $N$), that is, 
	 	if $u:\;\left[0,a\right]\to U$ (resp. $u:\; \left[0,a\right]\to N$)
	 	is a solution of $\pi$, then $t\mapsto g^+(u(t))$ (resp. $t\mapsto g^-(u(t))$) is 
	 	continuous and monotone decreasing on $\left[0,a\right]$.
	 	
	 	\begin{lemma}
	 		\label{le:140129-1335}
	 		\begin{enumerate}
	 			\item[(a)] $g^+$ is lower-semicontinuous.
	 			\item[(b)] $g^-$ is lower-semicontinuous.
	 			\item[(c)] $\{g^+\leq c\} := \{(y,x)\in N:\; g^+(y,x)\leq c\}$ is closed.
	 			\item[(d)] $\{g^-\leq c\} := \{(y,x)\in N:\; g^-(y,x)\leq c\}$ is closed.
	 			\item[(e)] For all $c_1, c_2> 0$, the set
	 			$(\{g^- \leq c_1\}\cap \{g^+>c_2\})$ is a neighborhood of $K := \Inv(N)$.
	 		\end{enumerate}
	 	\end{lemma}
	 	
	 	\begin{proof}
	 		\begin{enumerate}
	 			\item[(a)]
	 			Let $\eps>0$ and $(y,x)\in \Hull^+(y_0)\times X$. Suppose that $(y_n,x_n)\to (y,x)$ in $\Hull^+(y_0)\times X$
	 			and $g^+(y_n,x_n)\leq g^+(y,x) - \eps$ for all $n\in\IN$.
	 			We can assume w.l.o.g. that $g^+(y_n,x_n)\to t_0$.
	 			
	 			First of all, as $N$ is strongly admissible and $(y_n,x_n)\pi s\to (y,x)\pi s$, it follows that
	 			$(y,x)\pi s\in N$ for all $s\in\left[0,t_0\right]$. Secondly, 
	 			one has $(y_n,x_n)\pi g^+(y_n,x_n)\in X\setminus U$, which 
	 			is closed, so $(y,x)\pi t_0\in X\setminus U$. However, $t_0\leq g^+(y,x)-\eps$,
	 			which is a contradiction.
	 			\item[(b)] Let $(y,x)\in \Hull^+(y_0)\times X$ and 
	 			suppose that $(y_n,x_n)\to (y,x)$ but
	 			$g^-(y_n,x_n)\leq g^-(y,x) - \eps$ for some $\eps>0$.
	 			
	 			Let $t\in\left[0,g^+(y,x)\right[$ be arbitrary.
	 			By the lower-semicontinuity of $g^+$, one has $g^+(y_n,x_n)\geq t$
	 			provided that $n$ is sufficiently large. Furthermore,
	 			$d((y,x)\pi t,\Inv^-(N)) \leq d((y,x)\pi t, (y_n,x_n)\pi t) + d((y_n,x_n)\pi t, \Inv^-(N))$,
	 			so
	 			$d((y,x)\pi t,\Inv^-(N) \leq g^-(y_n,x_n) \leq g^-(y,x) - \eps$. The last inequality
	 			holds for arbitrary $t\in\left[0,g^+(y,x)\right[$. We thus have
	 			$g^-(y,x)\leq g^-(y,x)-\eps$, which is a contradiction.
	 			\item[(c), (d)] This follows immediately from the lower-semicontinuity
	 			of the respective function.
	 			\item[(e)] Arguing by contradiction,
	 			we may assume that there are $(y_n,x_n)\to (y,x)\in K$ such that
	 			either $g^+(y_n,x_n)\leq c_2$ or $g^-(y_n,x_n)>c_1$
	 			for all $n\in\IN$.
	 			In the first case, it follows that $g^+(y,x)\leq c_2$ in contradiction
	 			to $(y,x)\in K$. In the second case, we can choose $t_n\in\IR^+$
	 			such that for all $n\in\IN$, $t_n\leq g^+(y_n, x_n)$ and
	 			\begin{equation}
	 			\label{eq:140128-1808}
	 			d((y_n,x_n)\pi t_n, \Inv^-(N))\geq c_1>0.
	 			\end{equation}
	 			Either $(t_n)_n$ has a convergent subsequence or $t_n\to\infty$.
	 			Suppose that $(t_{n(k)})_k$ is a subsequence with $t_{n(k)}\to t_0$
	 			as $k\to\infty$. It follows that $d((y,x)\pi t_0, \Inv^-(N))\geq c_1$,
	 			which is a contradiction to $(y,x)\in K$.
	 			Thus, one has $t_n\to\infty$, and using the admissibility
	 			of $N$, there is a subsequence
	 			$(y_{n(k)},x_{n(k)})\pi t_{n(k)}$ which converges to 
	 			a point $(y',x')\in\Inv^-(N)$, in contradiction to \eqref{eq:140128-1808}.
	 		\end{enumerate}
	 	\end{proof}
	 	
	 	\begin{lemma}
	 		\label{le:151216-1714}
	 		For $c_1>0$ and $c_2>0$, set
	 		\begin{equation*}
	 		\begin{split}
	 		L^{c_1,c_2}_1 &:= \{g^-\leq c_1\} \cap \cl\{g^+\geq c_2\}\\
	 		L^{c_1,c_2}_2 &:= L^{c_1,c_2}_1\cap \{g^+\leq c_2\}
	 		\end{split}
	 		\end{equation*}
	 		and $\hat L^{c_1,c_2}_i := r^{-1}(L^{c_1,c_2}_i)$, $i=1,2$.
	 		
	 		Then for $c_1$ small and $c_2$ large, one has
	 		\begin{enumerate}
	 			\item $L^{c_1,c_2}_1\subset U$, and
	 			\item 
	 			$(L_1, L_2) := (\hat L_1, \hat L_2):= (\hat L^{c_1,c_2}_1, \hat L^{c_1, c_2}_2)$ is an index pair for $(y_0, K)$.
	 		\end{enumerate}
	 	\end{lemma}
	 	
	 	\begin{proof}
	 		\begin{enumerate}
	 			\item If $(y,x)\in \cl\{g^+\geq c_2\}$, then $(y,x)\pi \left[0,c_2\right]\subset N$.
	 			Hence, if the claim does not hold, there is a point $(y',x')\in K\cap (N\setminus U) = \emptyset$.
	 			\item 
	 			\begin{enumerate}
	 				\item[(IP1)]
	 				It follows from Lemma \ref{le:140129-1335} (c) and (d) that $L^{c_1,c_2}_1$ and $L^{c_1,c_2}_2$
	 				are closed, so $\hat L_1$ and $\hat L_2$ are closed by the continuity of $r$. 
	 				\item[(IP2)]
	 				Let $x\in L^{c_1,c_2}_1$ and $x\pi t\not\in L^{c_1,c_2}_1$ for some $t\geq 0$. The semiflow
	 				does not explode in $N$. Hence, there is a $t'\leq t$ such that $x\pi t'\in (\Hull^+(y_0)\times X)\setminus L^{c_1, c_2}_1$.
	 				Choose a sequence $x_n\to x$ in $L^{c_1,c_2}_1$ with $g^+(x_n)\geq c_2$. We have $x_n\pi t\not\in L^{c_1,c_2}_1$
	 				for all $n$ sufficiently large, so $x_n\pi s_n\in L^{c_1,c_2}_2$ for some $s_n\leq t$ and all $n\in\IN$.
	 				We can assume w.l.o.g. that $s_n\to s_0\leq t$, so $x\pi s_0\in L^{c_1,c_2}_2$.
	 				\item[(IP3)]
	 				Let $x\in L^{c_1,c_2}_2$ and $x\pi\left[0,t\right]\subset L^{c_1,c_2}_1$.
	 				We have $L^{c_1,c_2}_1\subset U$, so $g^+(x\pi s)\leq g^+(x)$ for all $s\in\left[0,t\right]$.
	 				Hence, $x\pi \left[0,t\right]\subset L^{c_1,c_2}_2$.
	 			\end{enumerate}
	 			
	 			Furthermore, one has $N\supset L^{c_1,c_2}_1\setminus L^{c_1,c_2}_2 \supset W$, 
	 			where $W:= \{g^-\leq c_1\}\cap \{g^+>c_2\}$ is a neighborhood of $K$ by Lemma \ref{le:140129-1335} (e).
	 			Thus, $r^{-1}(N)\supset \hat L^{c_1,c_2}_1\setminus \hat L^{c_1,c_2}_2\supset r^{-1}(W)$, which shows
	 			that $(\hat L_1,\hat L_2)$ is an index pair for $(y_0, K)$.
	 		\end{enumerate}
	 	\end{proof}
	 	
	 	Until now, our proof is based loosely
	 	on the respective proof in \cite{ryb} concerning the
	 	existence of isolating blocks. However, our claim
	 	is significantly weaker, so the proof is - hopefully - easier to follow.
	 	
	 	Since both $(N_1, N_2)$ and $(M_1, M_2)$ are index pairs for $(y_0,K)$,
	 	we can assume without loss of generality that
	 	$r^{-1}(N)\subset N_1\cap M_1$. Otherwise, one can simply replace $N$ by
	 	a sufficiently small neighborhood $N'$, and thereby obtain a stronger result.
	 	In order to complete the proof of Theorem \ref{th:140128-1546},
	 	we need 
	 	\begin{lemma}
	 		\label{le:151216-1716}
	 		For every $d>0$, 
	 		one has $\hat L^{c,d}_2\subset N^{-T}_2$ (resp. $\hat L^{c,d}_2\subset M^{-T}_2$) provided that $c$ is sufficiently small
	 		and $T$ is sufficiently large.
	 	\end{lemma}
	 	
	 	\begin{proof}
	 		If the lemma is not true, then there are sequences $((t_n,x_n))_n$, $c_n\to 0$
	 		and $T_n\to\infty$ such that $(t_n,x_n)\in \hat L^{c_n,d}_2$
	 		and $(t_n,x_n)\pi s\in N_1\setminus N_2$ for all $s\leq T_n$
	 		and all $n\in\IN$.
	 		
	 		Taking subsequences and because $c_n\to 0$, we can assume without loss of generality
	 		that $(y^{t_n}_0,x_n)\to (y,x)\in \Inv^-(N)$, which is compact because
	 		$N$ is strongly admissible. Since $(N_1, N_2)$
	 		is an index pair for $K$, there exists an isolating neighborhood $\tilde N$
	 		for $K$ with $N_1\setminus N_2\subset r^{-1}_{y_0}(\tilde N)$.
	 		The choice of the sequences implies that $(y,x)\in \Inv^+(\tilde N)$,
	 		so $(y, x)\in \Inv(\tilde N) = K$.
	 		
	 		However, $(y^{t_n}_0,x_n)\pi g^+(y^{t_n}_0,x_n)\in N\setminus U$ for all $n\in\IN$.
	 		Furthermore, $g^+(y^{t_n}_0, x_n)\leq d$ by the choice of $\hat L^{c,d}_2$.
	 		One may therefore assume w.l.o.g. that $g^+(y^{t_n}_0,x_n)\to t_0$. Consequently, one obtains
	 		$(y,x)\pi t_0\in (N\setminus U)\cap K = \emptyset$, which is an obvious contradiction.
	 	\end{proof}
	 	
	 	By using Lemma \ref{le:151216-1714}, 
	 	one can construct an index pair $(L_1, L_2) := (\hat L^{c,d}_1, \hat L^{c,d}_2)$ for $(y_0, K)$
	 	choosing $c$  small and $d$ large.
	 	In view of Lemma \ref{le:151216-1716}, one can find a possibly even smaller parameter
	 	$c>0$ such that the conclusions of Theorem \ref{th:140128-1546} hold for large $t_0$.
	 	The proof of Theorem \ref{th:140128-1546} is complete. 
	 \end{section}
	 
	 \begin{section}{Categorial Conley index}
	 	A connected simple system is a small category with the following property:
	 	if $A$ and $B$ are objects, then there is exactly one morphism $A\to B$.
	 	
	 	Understanding the Conley index as a connected simple system is
	 	perhaps the most elegant variant of the index. There is no loss of information,
	 	and other invariants such as a homotopy or (co)homology
	 	index can be derived by applying an appropriate functor.
	 	We will show in this section, that the nonautonomous
	 	extension of the Conley index defines a connected simple system as well.
	 	
	 	
	 	Throughout this section,
	 	we will assume the hypotheses\footnote{i.e. the spaces $X$, $Y$, the semiflows $\pi$, $\chi_{y_0}$ and the
	 		mapping $r:=r_{y_0}$}	 		
	 	at the beginning of the previous section.
	 	
	 	\begin{definition}
	 		\label{df:catci}
	 		\index{$\Con(y_0,K)$}
	 		\index{categorial Conley index (nonautonomous)}
	 		Let $y_0\in Y$, and let $K\subset \Hull^+(y_0)\times X$ be an isolated invariant set
	 		admitting a strongly admissible isolating neighborhood. The categorial
	 		(nonautonomous) Conley index $\Con(y_0,K)$ of $(y_0,K)$ is the smallest subcategory
	 		of the homotopy category of pointed spaces with the following properties:
	 		\begin{enumerate}
	 			\item Objects of $\Con(y_0,K)$ are pairs $(N_1/N_2, N_2)$, where
	 			$(N_1, N_2)$ is an index pair for $(y_0, K)$.
	 			\item If $(N_1, N_2)$ and $(M_1, M_2)$ are index pairs for $(y_0,K)$ with $(N_1, N_2)\subset (M_1, M_2)$,
	 			then the inclusion induced morphism $i:\; (N_1/N_2, N_2)\to (M_1/M_2, M_2)$ in the homotopy category
	 			of pointed spaces is a morphism of $\Con(y_0, K)$.
	 		\end{enumerate}
	 		\index{$\rep{N_1, N_2}$}
	 		For brevity, we also write $\rep{N_1,N_2} := (N_1/N_2, N_2)$.
	 	\end{definition}
	 	
	 	\begin{theorem}
	 		\label{th:131121-1208}
	 		$\Con(y_0,K)$ is (well-defined and) a connected simple system.
	 	\end{theorem}
	 	
	 	The proof below can be sketched as follows:
	 	Given two arbitrary index pairs $(N_1, N_2)$ and $(M_1, M_2)$, one constructs
	 	a morphism $f:\; \rep{N_1, N_2}\to \rep{M_1, M_2}$ in $\Con(y_0, K)$. This morphism
	 	$f$ is a composition of inclusion induced morphisms or their inverse morphisms and therefore
	 	necessarily a morphism of $\Con(y_0,K)$.
	 	These morphisms are then shown to be unique, that is, $f$
	 	depends only on $(N_1, N_2)$ and $(M_1, M_2)$, and invariant with respect
	 	to composition.
	 	In other words, the proof is nothing but an explicit construction.
	 	
	 	\begin{proof}
	 		Let $(N_1, N_2)$ and $(M_1, M_2)$ be arbitrary index pairs for $(y_0,K)$.
	 		By Theorem \ref{th:140128-1546}, there is an index pair $(L_1, L_2)$ for $(y_0,K)$
	 		and a $T\in\IR^+$ such that
	 		\begin{equation*}
	 		(L_1, L_2)\subset (N_1\cap M_1, N^{-T}_2 \cap M^{-T}_2).
	 		\end{equation*}
	 		
	 		Each inclusion of index pairs gives rise to a morphism. We obtain
	 		the following diagram, the arrows of which denote isomorphisms (Lemma \ref{le:140923-2008})
	 		(respectively the inverse morphim) of $\Con(y_0,K)$.
	 		\begin{equation}
	 		\label{eq:131121-1315}
	 		\xymatrix@1{
	 			\rep{N_1,N_2} \ar[r] & \rep{N_1, N^{-T}_2} & \ar[l] \rep{L_1, L_2} \ar[r] & \rep{M_1,M^{-T}_2} & \ar[l] \rep{M_1,M_2}
	 		}
	 		\end{equation}
	 		It follows that there is a morphism in $\rep{N_1,N_2}\to \rep{M_1, M_2}$ in 
	 		$\Con(y_0, K)$, namely the composition of the morphisms
	 		in the row above.
	 		
	 		Next, we will show that the morphism obtained using this procedure is
	 		unique. Firstly, let $T_1\geq T_2$ be positive real numbers. The following
	 		ladder with inclusion induced arrows is commutative.
	 		\begin{equation*}
	 		\xymatrix{
	 			\rep{N_1,N_2} \ar[r] & \rep{N_1, N^{-T_1}_2} & \ar[l] \rep{L_1, L_2} \ar[r] & \rep{M_1,M^{-T_1}_2} & \ar[l] \rep{M_1,M_2}\\
	 			\rep{N_1,N_2} \ar[r] \ar[u] & \rep{N_1, N^{-T_2}_2} \ar[u] & \ar[l] \ar[u] \rep{L_1, L_2} \ar[r] 
	 			& \ar[u] \rep{M_1,M^{-T_2}_2} & \ar[l] \ar[u] \rep{M_1,M_2}
	 		}
	 		\end{equation*}
	 		Hence, the morphism $\rep{N_1,N_2}\to \rep{M_1,M_2}$ defined by
	 		\eqref{eq:131121-1315} is independent of $T$. Secondly, one needs to
	 		consider the index pair $(L_1, L_2)$. Suppose $(L'_1, L'_2)$ is 
	 		another index pair for $(y_0,K)$ with $(L'_1, L'_2)\subset (N_1\cap M_1, N^{-T}_2\cap M^{-T}_2)$.
	 		It follows again from Theorem \ref{th:140128-1546} that there exist
	 		an index pair $(L''_1, L''_2)$ for $(y_0,K)$ and a constant $T>0$
	 		such that $(L''_1, L''_2)\subset (L_1\cap L'_1, L^{-T}_2\cap (L')^{-T}_2)$.
	 		
	 		We obtain a commutative diagram below, where each arrow denotes an inclusion
	 		induced (iso)morphism.
	 		\begin{equation*}
	 		\xymatrix{
	 			&& \rep{L_1,L_2} \ar[ldd] \ar[rdd] \ar[d]\\
	 			&& \rep{L_1,L^{-T}_2} \ar[ld] \ar[rd]\\
	 			\rep{N_1,N_2} \ar[r]
	 			&\rep{N_1,N^{-2T}_2}
	 			& \ar[l]
	 			\rep{L''_1,L''_2} \ar[u] \ar[d]
	 			\ar[r]
	 			&\rep{M_1,M^{-2T}_2}
	 			& \ar[l]
	 			\rep{M_1,M_2} \\
	 			&& \rep{L'_1,(L'_2)^{-T}} \ar[lu] \ar[ru]\\
	 			&& \rep{L'_1,L'_2 \ar[u]} \ar[luu] \ar[ruu].
	 		}
	 		\end{equation*}
	 		The morphisms defined by $(L_1,L_2)$ and $(L'_1, L'_2)$
	 		agree since each arrow in the above diagram denotes an 
	 		isomorphism (Lemma \ref{le:140923-2008}).
	 		
	 		Finally, we will show that the composition of two morphisms
	 		obtained from the above prodecure can be written as in 
	 		\eqref{eq:131121-1315}. Suppose, we are given index pairs
	 		$(N_1, N_2)$, $(M_1, M_2)$ and $(O_1,O_2)$ for $(y_0,K)$. 
	 		By Theorem \ref{th:140128-1546}, there are an index pair
	 		$(L_1, L_2)$ for $(y_0, K)$ and a $T\in\IR^+$ such that
	 		\begin{equation*}
	 		(L_1, L_2) \subset (N_1\cap M_1\cap O_1, N^{-T}_2\cap M^{-T}_2\cap O^{-T}_2).
	 		\end{equation*}
	 		
	 		\newcommand{\ot}{\leftarrow}
	 		For every two objects $A,B$ in $\Con(y_0,K)$,
	 		let $A\to B$ denote the unique morphism defined
	 		by \eqref{eq:131121-1315}. We also write $B\ot A$
	 		for the inverse (morphism) of $A\to B$.
	 		Given morphisms $A\to B$ and
	 		$B\to C$, we write $A\to B\to C$ to denote their composition.
	 		We need to prove that $A\to B\to C = A\to C$. One has
	 		\begin{align*}
	 		&\rep{N_1,N_2}\to \rep{M_1,M_2}\to \rep{O_1,O_2}\\
	 		=\,& \rep{N_1,N_2}\to \rep{N_1,N^{-T}_2}\ot \rep{L_1,L_2}
	 		\to \rep{M_1,M^{-T}_2}\ot \rep{M_1,M_2}\\
	 		&\to \rep{M_1, M^{-T}_2} \ot \rep{L_1,L_2} \to \rep{O_1,O^{-T}_2} \ot \rep{O_1,O_2}\\
	 		=\,& \rep{N_1,N_2}\to \rep{N_1,N^{-T}_2}\ot \rep{L_1,L_2}
	 		\to \rep{M_1,M^{-T}_2}
	 		\ot \rep{L_1,L_2} \to \rep{O_1,O^{-T}_2}\\ & \ot \rep{O_1,O_2}\\
	 		=\,& \rep{N_1, N_2} \to \rep{N_1, N^{-T}_2} \ot \rep{L_1,L_2} \to \rep{O_1,O^{-T}_2} \ot \rep{O_1,O_2}\\
	 		=\,& \rep{N_1,N_2} \to \rep{O_1,O_2}.
	 		\end{align*}
	 	\end{proof}
	 	
	 	\index{connected simple system}
	 	\index{$\CSS(\mathcal{K})$}
	 	We will now introduce $\CSS(\mathcal{K})$, the category of connected simple systems
	 	in a given category $\mathcal{K}$. Objects of $\CSS(\mathcal{K})$ are subcategories of $\mathcal{K}$
	 	which are connected simple systems. Let $\mathcal{A}$ and $\mathcal{B}$ be connected
	 	simple systems in $\mathcal{K}$. A morphism $\mathcal{A}\to \mathcal{B}$ in 
	 	$\CSS(\mathcal{K})$ is a family $(f_{A,B})_{(A,B)\in \Obj(\mathcal{A})\times \Obj(\mathcal{B})}$, 
	 	where $\Obj(.)$ denotes the objects of a given category and each $f_{A,B}$ is a morphism $A\to B$ in $\mathcal{K}$ such that 
	 	\begin{equation*}
	 	\xymatrix{
	 		A \ar[r]^-{f_{A,B}} \ar[d] &B \ar[d]\\
	 		A' \ar[r]^-{f_{A',B'}} & B' \\
	 	}
	 	\end{equation*}
	 	is commutative. The vertical arrows denote the unique (inner) morphisms in
	 	$\mathcal{A}$ respectively $\mathcal{B}$. 
	 	
	 	If $A$ is an object of $\mathcal{A}$, $B$ is an object of $\mathcal{B}$,
	 	and $f:\;A\to B$ is a morphism, then there is a unique morphism $F\in \CSS(\mathcal{K})$
	 	with $F=F(A,B) = f$. We say that $\rep{f} := F$ is induced by $f$.
	 	
	 	\index{inner morphism}
	 	Now, set $\mathcal{K} = \HTop$, the homotopy category of pointed spaces,
	 	and given an isolated invariant set $K\subset \Hull^+(y_0)\times X$
	 	admitting a strongly admissible isolating neighborhood, its index
	 	$\Con(y_0, K)$ is an object of $\CSS(\HTop)$. The morphisms of $\Con(y_0,K)$
	 	are called {\em inner morphisms}.
	 \end{section}
	 
	 \begin{section}{Homology Conley index and attractor-repeller sequences}
	 	In this section, attractor-repeller decompositions of isolated invariant sets 
	 	are studied. The main tool are long exact sequences in homology. 
	 	
	 	\begin{subsection}{Attractor-repeller decompositions and index triples}
	 		Attractor-repeller decompositions with respect to semiflows are not exactly a new concept; in particular
	 		since they are applied to the skew-product formulation of the nonautonomous problem.
	 		The main goal of this section is to understand the implications of having an
	 		attractor-repeller decomposition in a space $\Hull^+(y_0)\times X$ on the index pairs
	 		respectively the index, living in the space $\IR^+\times X$.
	 		
	 		First of all, $\alpha$ and $\omega$-limes sets
	 		can be defined as usual.
	 		\begin{align*}
		 		\alpha(u) &:= \bigcap_{t\in\IR^-} \cl_{\Hull^+(y_0)\times X} u(\left]-\infty,t\right])\\
		 		\omega(u) &:= \bigcap_{t\in\IR^+} \cl_{\Hull^+(y_0)\times X} u(\left[t,\infty\right[)
	 		\end{align*}
	 		
	 		Based on the above definitions, the notion of an attractor-repeller decomposition
	 		can be made precise.
	 			 		
	 		\begin{definition}
	 			\index{attractor-repeller decomposition}
	 			Let $y_0\in Y$ and $K\subset \Hull^+(y_0)\times X$ be an isolated
	 			invariant set. $(A,R)$ is an {\em attractor-repeller decomposition} of 
	 			$K$ if $A,R$ are disjoint isolated invariant subsets of $K$ and for every solution $u:\;\IR\to K$
	 			one of the following alternatives holds true.
	 			\begin{enumerate}
	 				\item $u(\IR)\subset A$
	 				\item $u(\IR)\subset R$
	 				\item $\alpha(u)\subset R$ and $\omega(u)\subset A$
	 			\end{enumerate}
	 			We also say that $(y_0, K, A, R)$ is an attractor-repeller decomposition.
	 		\end{definition}
	 		
	 		\begin{definition}
	 			\label{df:index_triple}
	 			\index{index triple (for $(y_0,K,A,R)$)}
	 			Let $y_0\in Y$ and $K\subset \Hull^+(y_0)\times X$ be an isolated
	 			invariant set admitting a strongly admissible isolating neighborhood $N$.
	 			Suppose that $(A,R)$ is an attractor-repeller decomposition of 
	 			$K$.
	 			
	 			A triple $(N_1, N_2, N_3)$ is called an {\em index triple} for
	 			$(y_0,K,A,R)$ provided that:
	 			\begin{enumerate}
	 				\item $N_3\subset N_2\subset N_1$
	 				\item $(N_1, N_3)$ is an index pair for $(y_0,K)$
	 				\item $(N_2, N_3)$ is an index pair for $(y_0,A)$
	 			\end{enumerate}
	 		\end{definition}
	 		
	 		Suppose we are given an isolated invariant set and an attractor-repeller
	 		decomposition thereof. Does there exist an index triple?
	 		
	 		\begin{lemma}
	 			\label{le:141205-1638}
	 			\index{index triple, existence}
	 			Let $y_0\in Y$ and $K\subset \Hull^+(y_0)\times X$ be an isolated
	 			invariant set admitting a strongly admissible isolating neighborhood $N$.
	 			Suppose that $(A,R)$ is an attractor-repeller decomposition of 
	 			$K$.
	 			
	 			Then there exists an index triple $(N_1, N_2, N_3)$ for $(y_0,K,A,R)$
	 			such that $N_1\subset r^{-1}(N)$.
	 		\end{lemma}
	 		
	 		\begin{proof}
	 			It is known that there exists an FM-index triple $(N'_1, N'_2, N'_3)$ (see \cite{hib})
	 			with $N_1\subset N$. By Lemma \ref{le:141113-1800}, $(r^-{1}(N'_1),r^{-1}(N'_3))$ is an 
	 			index pair for $(y_0, K)$ and $(r^{-1}(N'_2), r^{-1}(N'_3))$
	 			is an index pair for $(y_0, A)$.
	 		\end{proof}
	 		
	 		\begin{lemma}
	 			\label{le:141106-1717}
	 			Let $(N_1, N_2, N_3)$ be an index triple for $(y_0, K, A, R)$.

	 			Then, $(N_1, N_2)$ is an index pair for $(y_0, R)$.
	 		\end{lemma}
	 		
	 		\begin{proof}
	 			Firstly, we will show that $(N_1, N_2)$ is a basic index pair,
	 			that is, we need to check Definition \ref{df:basic_index_pair}.
	 			\begin{enumerate}
	 				\item[(IP2)] Let $x\in N_1$ and $t\in\IR^+$ such 
	 				that $x\chi_{y_0} t\not\in N_1$. It is known that $(N_1, N_3)$
	 				is an index pair, so $x\chi_{y_0} s\in N_3\subset N_2$ for some $s\in\left[0,t\right]$.
	 				\item[(IP3)] Let $x\in N_2$ and $t\in\IR^+$ such 
	 				that $x\chi_{y_0} t\not\in N_2$. $(N_2, N_3)$ is an index pair,
	 				so $x\chi_{y_0} s\in N_3$ for some $s\in\left[0,t\right]$.
	 				Since $(N_1, N_3)$ is also an index pair, it follows
	 				that $x\chi_{y_0} s'\in X\setminus N_1$ for some $s'\in\left[s,t\right]$.
	 			\end{enumerate}
	 			
	 			Recall the mapping $r:=r_{y_0}$, which can be found in Definition \ref{df:140128-1503}.
	 			Since $(N_1, N_3)$ (resp. $(N_2, N_3)$) is an index pair
	 			for $(y_0, K)$ (resp. $(y_0, A)$), there is a strongly admissible
	 			isolating neighborhood $M_K$ (resp. $M_A$) such that
	 			$N_1\setminus N_2\subset r^{-1}(M_K)$ (resp. $N_2\setminus N_3\subset r^{-1}(M_A)$).
	 			There also exists an open neighborhood $W_K$ (resp. $W_A$) of $K$ (resp. $A$)
	 			with $r^{-1}(W_K)\subset N_1\setminus N_3$ (resp. $r^{-1}(W_A)\subset N_2\setminus N_3$).
	 			
	 			Recall that $A\cap R=\emptyset$ by the definition of an attractor-repeller decomposition,
	 			so there are disjoint open neighborhoods $U_A$ of $A$ and $U_R$ of $R$. We may assume without
	 			loss of generality that $W_A\subset U_A$.
	 			Setting $M_R := M_K\setminus W_A$, one has $\Inv M_R \subset R \subset U_R \subset M_R$, which means that $M_R$ is
	 			an isolating neighborhood for $R$.
	 			
	 			Moreover, one has 
	 			\begin{equation*}
	 			N_1\setminus N_2 = (N_1\setminus N_3)\setminus(N_2\setminus N_3) \subset r^{-1}(M_K)\setminus r^{-1}(W_A) = r^{-1}(M_R).
	 			\end{equation*}
	 			
	 			Define $N'_A := \cl_{\Hull^+(y_0)\times X} r(N_2\setminus N_3)$ and $W_R:= W_K\setminus N'_A$. One has
	 			\begin{equation*}
	 			N_1\setminus N_2 \supset r^{-1}(W_K) \setminus (N_2\setminus N_3) \supset r^{-1}(W_K)\setminus r^{-1}(N'_A) = r^{-1}(W_R).
	 			\end{equation*}
	 			The set $K\cap N'_A \subset M_A$ is positively invariant: Let $x\in K\cap N'_A$
	 			and $x\pi s\in K\setminus N'_A$ for some $s\in\IR^+$.
	 			There is a sequence $(t_n,x_n)$ in $N_2\setminus N_3\subset \IR^+\times X$
	 			such that $r(t_n,x_n)\to x$ as $n\to\infty$. We can assume
	 			that $r(t_n,x_n)\pi s\not\in N'_A$ for all $n\in\IN$,
	 			so w.l.o.g. there are reals $s_n\to s_0$ with $(t_n,x_n)\chi_{y_0} s_n\in N_3$ for all $n\in\IN$.
	 			We have $r(t_n,x_n)\pi s_n\to x\pi s_0\in K$, so $(t_n,x_n)\chi_{y_0} s_n\in r^{-1}(W_K)$
	 			for all but finitely many $n$, which is a contradiction since $r^{-1}(W_K)\cap N_3 = \emptyset$.
	 			
	 			Hence, if $x\in K\cap N'_A$, 
	 			then $\omega(x)\subset A$, implying that $R\cap N'_A=\emptyset$.
	 			Therefore $W_R$, which is obviously open, is a neighborhood of $R$.
	 		\end{proof}
	 		
	 		\begin{lemma}
	 			\label{le:141029-1824}
	 			Let $(N_1, N_2, N_3)$ be an index triple for $(y_0, K, A, R)$. Then, for every $T\in\IR^+$
	 			\begin{equation*}
	 			(N_1, N^{-T}_2, N_3) := (N_1, N^{-T}_2(N_1), N_3)
	 			\end{equation*}
	 			and
	 			\begin{equation*}
	 			(N_1, N^{-T}_2, N^{-T}_3) := (N_1, N^{-T}_2(N_1), N^{-T}_3(N_1))
	 			\end{equation*}
	 			are index triples for $(y_0, K, A, R)$.
	 		\end{lemma}
	 		
	 		\begin{proof}
	 			Lemma \ref{le:140924-1939} implies that $(N_1, N^{-T}_2)$ and
	 			$(N_1, N^{-T}_3)$ are index pairs for $(y_0, K)$ for every $T>0$. Furthermore,
	 			assuming that $(N^{-T}_2, N_3)$ is an index pair for $(y_0, A)$, it follows
	 			from Lemma \ref{le:140924-1939}\footnote{$N^{-T}_3(N_1)=N^{-T}_3(N^{-T}_2)$} that $(N^{-T}_2, N^{-T}_3)$ is
	 			an index pair for $(y_0, A)$.
	 			
	 			Hence, we only need to prove
	 			that $(N^{-T}_2, N_3)$ is an index pair for $(y_0, A)$. 
	 			\begin{enumerate}
	 				\item[(IP1)] $(N_1, N^{-T}_2)$ is an index pair, so $N^{-T}_2$ is closed.
	 				\item[(IP2)] Let $x\in N^{-T}_2$ and $x\chi_{y_0} t\not\in N^{-T}_2\supset N_2$.
	 				We have $x\chi_{y_0} s'\in N_2$ for some $s'\leq t$.
	 				Since $(N_2, N_3)$ is an index pair, we must have $x\chi_{y_0} s\in N_3$ for some $s\in\left[s',t\right]$.
	 				\item[(IP3)] Let $x\in N_3$ and $x\chi_{y_0} t\not\in N_3$.
	 				$(N_1, N_3)$
	 				is an index pair, so $x\chi_{y_0} s\in (\IR^+\times X)\setminus N_1\subset (\IR^+\times X)\setminus N^{-T}_2$ for some $s\in\left[0,t\right]$.

					\item[(IP4)]
		 			$(N_1, N^{-T}_2)$ is an index pair for $(y_0, R)$, so there is an open neighborhood $W_R$ 
		 			of $R$ such that $r^{-1}(W_R)\subset N_1\setminus N^{-T}_2$. We may assume that $W_R\cap A=\emptyset$
		 			because $A\cap R=\emptyset$. Let $N_K$ be an isolating
		 			neighborhood for $K$ with $N_1\setminus N_3\subset r^{-1}(N_K)$. Then $N_A := N_K\setminus W_R$
		 			is an isolating neighborhood for $A$ with 
		 			\begin{equation*}
		 			N^{-T}_2\setminus N_3 \subset (N_1\setminus N_3)\setminus (N_1\setminus N^{-T}_2) \subset r^{-1}(N_A).
		 			\end{equation*}
			 		
			 		\item[(IP5)]
		 			Since $(N_2, N_3)$ is an index pair for $(y_0, A)$,
		 			there is a neighborhood $W_A$ of $A$ with $r^{-1}(W_A)\subset N_2\setminus N_3$. 
		 			One has 
		 			\begin{equation*}
		 			r^{-1}(W_A)\subset N_2\setminus N_3 \subset N^{-T}_2(N_1)\setminus N_3.
		 			\end{equation*}
		 		\end{enumerate}
	 		\end{proof}
	 	\end{subsection}
	 	
	 	\begin{subsection}{Long exact sequences}
	 		The long exact sequence associated with an attractor-repeller sequence
	 		is usually defined using the concept of so-called weakly exact sequences
	 		(Definition 2.1 in \cite{connmatrix}).
	 		Instead of weakly exact sequences, we use the long exact sequence 
	 		of a triple as a starting point. The
	 		advantage is that our definition relies only on an axiomatic characterization 
	 		of homology yet not necessarily on an underlying chain complex.
	 		It is therefore only assumed that $\Hom_* = (\Hom_q)_{q\in\IZ}$ is
	 		a homology theory satisfying the Eilenberg--Steenrod axioms. Of course,
			$\Hom_*$ can also simply be read as the singular homology functor.
	 		
	 		\begin{lemma}
	 			\label{le:140925-1601}
	 			Let $(N_1, N_2, N_3)$ be an index triple for $(y_0,K,A,R)$.
	 			Then, the projection $p:\; N_1/N_3 \to N_1/N_2$
	 			induces an isomorphism $\varrho:\;\Hom_*(N_1/N_3, N_2/N_3)\to \Hom_*(N_1/N_2, \{N_2\})$.
	 		\end{lemma}
	 		
	 		The proof will be conducted in three steps, the first two being formulated as separate lemmas.
	 		
	 		\begin{lemma}
	 			\label{le:140925-1628}
	 			Let $(N_1, N_2)$ be an index pair for $(y_0, K)$ and define
	 			$f:\; N_1\to \IR^+$ by 
	 			\begin{equation*}
	 			f(t,x) := \sup\{ t_0\in\IR^+:\; (t,x)\chi_{y_0} s\in \cl (N_1\setminus N_2) \text{ for all }s\in\left[0,t_0\right]\}.
	 			\end{equation*}
	 			Then,
	 			\begin{enumerate}
	 				\item[(a)] $f$ is upper semicontinuous and
	 				\item[(b)] bounded on $N_2$.
	 			\end{enumerate}
	 		\end{lemma}
	 		
	 		\begin{proof}
	 			\begin{enumerate}
	 				\item[(a)] Suppose that $f$ is not upper semicontinuous. Then there is
	 				a sequence $(t_n,x_n)\to (t_0,x_0)$ in $N_1$ such that $f(t_n,x_n)>f(t_0,x_0)+\eps$
	 				for some $\eps>0$ and all $n\in\IN$.
	 				By the definition of $f$, there is an $s\in\left[0,\eps\right[$ with
	 				$(t_0, x_0)\chi_{y_0} (f(t_0,x_0) + s)\in (\IR^+\times X)\setminus (\cl (N_1\setminus N_2))$.
	 				It follows that $(t_n,x_n)\chi_{y_0} (f(t_0,x_0) + s) \in (\IR^+\times X)\setminus (\cl (N_1\setminus N_2))$
	 				for all $n$ sufficiently large. Hence, $f(t_n, x_n) < f(t_0,x_0) + \eps$ for those $n$, 
	 				which is a contradiction.
	 				\item[(b)] $(N_1, N_2)$ is an index pair for $(y_0,K)$, so there is a strongly
	 				admissible isolating neighborhood $N\subset\Hull^+(y_0)\times X$ for $K$
	 				such that $N_1\setminus N_2 \subset r^{-1}(N)$. $N$ is closed, so
	 				$\cl (N_1\setminus N_2)\subset r^{-1}(N)$. Furthermore, there exists an open neighborhood $W$ of $K$
	 				with $r^{-1}(W)\subset N_1\setminus N_2$. Now, suppose that $f$ is unbounded
	 				on $N_2$. Then there is a sequence $(t_n,x_n)$ in $N_2$ with $f(t_n, x_n)\to\infty$. 
	 				
	 				Because $f((t_n,x_n)\chi_{y_0}s)\neq 0$, we must have $(t_n,x_n)\chi_{y_0} s \in N_2\cap (\cl_{\IR^+\times X} (N_1\setminus N_2))$
	 				for all $s\in\left[0,f(t_n,x_n)\right[$ and all $n\in\IN$, so $r(t_n, x_n)\pi s\in N\setminus W$ for all $s\in\left[0,f(t_n,x_n)\right]$.
	 				
	 				Since $N$ is strongly admissible, there is a solution $u:\;\IR\to N\setminus W$ of $\pi$. 
	 				However, $u(\IR)\subset K$ because $N$ is an isolating neighborhood for $K$. This is a
	 				contradiction since $K\subset W$.
	 			\end{enumerate}
	 		\end{proof}
	 		
	 		\begin{lemma}
	 			\label{le:140925-1640}
	 			Let $(N_1, N_2)$ be an index pair for $(y_0, K)$. Then
	 			for all $T\in\IR^+$ sufficiently large, $N^{-T}_2 := N^{-T}_2(N_1)$ is 
	 			a neighborhood of $N_2$ in $N_1$.
	 		\end{lemma}
	 		
	 		\begin{proof}
	 			By Lemma \ref{le:140925-1628} (a), $W^T := f^{-1}(\left[0,T\right[)$ is open for every $T\in\IR^+$. 
	 			If $T$ is sufficiently large, then $W^T\supset N_2$ by Lemma \ref{le:140925-1628} (b), so $W^T$ is a neighborhood of $N_2$ in $N_1$. 
	 			We are going to show that $W^T\subset N^{-T}_2$,
	 			which implies that for large $T\in\IR^+$, $N^{-T}_2$ is a neighborhood of $N_2$ as claimed.
	 			
	 			In order to prove the inclusion $W^T\subset N^{-T}_2$, let $x\in W^T$ and $\eps>0$ be arbitrary. 
	 			We have $x\chi t\not\in \cl (N_1\setminus N_2)$ for some $t\leq T+\eps$ solely by the definition of $f$. 
	 			Either $x\chi_{y_0} t\in N_1$ and thus $x\chi_{y_0} t\in N_2$ or $x\chi t'\in N_2$ for some $t'\leq t$ because
	 			$(N_1, N_2)$ is an index pair. Since
	 			$\eps>0$ is arbitrary and $N_2$ closed, it follows that $x\chi t''\in N_2$ for some $t''\leq T$,
	 			so $x\in N^{-T}_2$.
	 		\end{proof}
	 		
	 		\begin{proof}[of Lemma \ref{le:140925-1601}]
	 			Consider the following sequence of inclusion induced mappings.
	 			\begin{equation*}
	 			\xymatrix{
	 				\Hom_*(N_1/N_3, N_2/N_3) \ar[r]^-i & 
	 				\Hom_*(N_1/N_3, N^{-T}_2/N_3)\\
	 				\ar[r]^-k &\Hom_*(N_1/N_2, N^{-T}_2/N_2) \ar[r]^-l &
	 				\Hom_*(N_1/N_2, N_2/N_2).
	 			}
	 			\end{equation*}
	 			We will show that $i,k,l$ are isomorphisms.
	 			
	 			Firstly, we consider $i$.
	 			Define $\varphi_T:\; N_1/N_3\to N_1/N_3$
	 			by
	 			\begin{equation*}
	 			\varphi_T(\rep{t,x}) := 
	 			\begin{cases}
	 			\rep{(t,x)\chi_{y_0} T} & (t,x)\chi_{y_0}s\in N_1\setminus N_3\text{ for all }s\in\left[0,T\right]\\
	 			N_3 & \text{otherwise}.
	 			\end{cases}
	 			\end{equation*}
	 			It follows from Lemma 3.7 in \cite{article_naci} that
	 			$\varphi_T$ and therefore its restriction 
	 			to $N^{-T}_2/N_3$ are continuous. We conclude that $i^{-1} = \varphi_T$
	 			up to homotopy, so $i$ is indeed an isomorphism.
	 			
	 			Secondly, choosing $T$ sufficiently large, it follows
	 			from Lemma \ref{le:140925-1640} that $N^{-T}_2$ is a neighborhood
	 			of $N_2\supset N_3$. Hence, $k$ is an isomorphism by the
	 			excision property of homology.
	 			
	 			Thirdly, it follows as before that the one-point space
	 			$N_2/N_2$ is a deformation retract of $N^{-T}_2/N_2$. Hence, $k$
	 			must be an isomorphism as well, completing the proof.
	 		\end{proof}
	 		
	 		In view of Lemma \ref{le:140925-1601}, we can now
	 		make define long exact sequences associated with index triples. To keep the definition short,
	 		recall that the homology theory defines a boundary operator (connecting
	 		homomorphism) $\partial(X,A)$ for every topological pair $(X,A)$.
	 		Let $(X,A,B)$ be a triple of topological spaces, where $B\subset A\subset X$
	 		are subspaces. There is a long exact sequence associated with $(X,A,B)$
	 		and its (natural) connecting homomorphism $\delta$ is given by $\delta := \Hom_*(k)\circ \partial(X,A)$,
	 		where $k:\; (A,\emptyset) \to (A,B)$ denotes the inclusion (see \cite[Theorem 5 in Section 4.8]{spanier}).
	 		
	 		\index{long exact sequence associated with an index pair}
	 		\begin{definition}
	 			\label{df:141024-1736}
	 			Let $(N_1, N_2, N_3)$ be an index triple for $(y_0, K, A, R)$. Let
	 			$q:\; \Hom_*(N_1/N_3, N_2/N_3)\to \Hom_*(N_1/N_2, N_2/N_2)$ be inclusion induced
	 			and set $\partial = \delta\circ q^{-1}$, where 
	 			$\delta$ is the connecting homomorphism
	 			associated with the triple $(N_1/N_3, N_2/N_3, N_3/N_3)$.
	 			
	 			The long exact sequence associated with $(N_1, N_2, N_3)$ is
	 			\begin{equation}
	 			\label{eq:140925-1842}
	 			\xymatrix@1{
	 				\ar[r] & \Hom_*\rep{N_2, N_3} \ar[r]^-i & \Hom_*\rep{N_1, N_3} \ar[r]^-p & \Hom_*\rep{N_1,N_2} \ar[r]^-\partial & \Hom_{*-1}\rep{N_2,N_3}
	 				\ar[r]&.
	 			}
	 			\end{equation}
	 		\end{definition}
	 		
	 		Here, we denote $\Hom_*\rep{N_1, N_2} := \Hom_*(N_1/N_2, \{N_2\})$.
	 		
	 		\begin{lemma}
	 			Let $(N_1, N_2, N_3)$ be an index triple for $(y_0, K, A, R)$.
	 			The sequence \eqref{eq:140925-1842} associated with $(N_1, N_2, N_3)$ is 
	 			exact.
	 		\end{lemma}
	 		
	 		\begin{proof}
	 			We rewrite \eqref{eq:140925-1842} as follows.
	 			\begin{equation*}
	 			\xymatrix{
	 				&&&\Hom_*\rep{N_1,N_2}  \ar[rd]^-{\partial}&\\
	 				\ar[r] & \Hom_*\rep{N_2, N_3} \ar[r]^-i & \Hom_*\rep{N_1, N_3} \ar[ru]^-p \ar[r] 
	 				& \Hom_*(N_1/N_3, N_2/N_3) \ar[u]^-q \ar[r]^-{\delta} & \Hom_{*-1}\rep{N_2,N_3}
	 				\ar[r]&
	 			}
	 			\end{equation*}
	 			The lower row is the long exact sequence of the triple $(N_1/N_3, N_2/N_3, N_3/N_3)$.
	 			The result follows easily because $q$ is an isomorphism by Lemma \ref{le:140925-1601}.
	 		\end{proof}
	 		
	 		\begin{lemma}
	 			\label{le:20141024-1820}
	 			Let $(N_1, N_2, N_3)$ (resp. $(N'_1, N'_2, N'_3)$) 
	 			be an index triple for $(y_0, K, A, R)$ (resp. $(y'_0, K', A', R')$).
	 			The boundary operator $\partial$ is natural with respect
	 			to continuous mappings $f:\; (N_1, N_2, N_3) \to (N'_1, N'_2, N'_3)$, that is,
	 			if $\partial$ and $\partial'$ denote the respective boundary operators, then
	 			\begin{equation*}
	 			\xymatrix{
	 				\Hom_*\rep{N_1, N_2} \ar[r]^-{\partial} \ar[d]^-f & \Hom_*\rep{N_2, N_3} \ar[d]^-f \\
	 				\Hom_*\rep{N'_1, N'_2} \ar[r]^-{\partial'} & \Hom_*\rep{N'_2, N'_3}
	 			}
	 			\end{equation*}
	 			is commutative.
	 		\end{lemma}
	 		
	 		\begin{proof}
	 			This follows easily from Definition \ref{df:141024-1736}.
	 			The connecting homomorphisms of the long exact sequences
	 			associated with a triple are natural, and so are the projections $q$.
	 		\end{proof}
	 	\end{subsection}
	 	
	 	\begin{subsection}{The homology Conley index}
	 		The homotopy index of an invariant set is the homotopy type
	 		of an appropriate quotient space. A homology index could be defined
	 		similarly as equivalence class of graded modules or the like.
	 		However, to not loose the connecting homomorphisms introduced in the
	 		previous sections, requires a more sophisticated approach. 
	 		
	 		\begin{definition}
	 			\index{homology Conley index}
	 			Let $y_0\in Y$ and $K\subset\Hull^+(y_0)\times X$ be
	 			an isolated invariant set admitting a strongly admissible
	 			isolating neighborhood, so its categorial Conley index $\Con(y_0,K)$
	 			is defined.
	 			
	 			The (categorial, nonautonomous) homology Conley index is
	 			obtained by applying the homology functor, that is:
	 			\begin{enumerate}
	 				\item If $A$ is an object of $\Con(y_0,K)$, then
	 				$\Hom_*(A)$ is an object of $\Hom_*\Con(y_0,K)$.
	 				\item If $f$ is a morphism of $\Con(y_0, K)$, then
	 				$\Hom_*(f)$ is a morphism of $\Hom_*\Con(y_0, K)$.
	 			\end{enumerate}
	 		\end{definition}
	 		
	 		As a consequence of the above definition, the homology Conley
	 		index is an object of $\CSS(\gradMod)$, where $\gradMod$
	 		denotes the category of graded modules. Suppose we are given
	 		a long sequence
	 		\begin{equation*}
		 		\xymatrix@1{
		 			\ar[r] & \mathcal{A}^n \ar[r]^-{\rep{f^n}} & \mathcal{A}^{n+1}  \ar[r] &
		 		}
		 	\end{equation*}
		 	in $\CSS(\mathrm{\gradMod})$. Choose objects $A^n$ of $\mathcal{A}^n$ for every $n\in\IN$.
		 	We say that the above sequence is exact if
		 	\begin{equation*}
		 		\xymatrix@1{
		 			\ar[r] & A^n \ar[r]^-{f^n} & A^{n+1}  \ar[r] &
		 		}
		 	\end{equation*}
		 	is exact. Note that this notion of exactness is independent of the particular choice of
		 	objects.
	 		
	 		The above definition immediately leads to
	 		the following question: Does a connecting homomorphism $\partial$
	 		which is defined by a particular index triple give rise
	 		to a unique morphism of the homology index? The answer is affirmative
	 		as we will see below.
	 		\index{attractor-repeller sequence}
	 		\index{connecting homomorphism}
	 		
	 		\begin{theorem}
	 			\label{th:150727-1913}
	 			Let $(N_1, N_2, N_3)$ be an index triple for $(y_0, K, A, R)$.
	 			The connecting homomorphism $\partial$ that is given by Definition
	 			\ref{df:141024-1736} gives rise to a unique i.e., independent of $(N_1, N_2, N_3)$,
	 			morphism $\rep{\partial}$
	 			in $\CSS(\gradMod)$ and 
	 			\begin{equation}
	 			\label{eq:160107-1253}
	 			\xymatrix@1{
	 				\ar[r] & \Hom_*\Con(y_0,A) \ar[r]^-{\rep{i}} & \Hom_*\Con(y_0,K) \ar[r]^-{\rep{p}} 
	 				& \Hom_*\Con(y_0,R) \ar[r]^-{\rep{\partial}} & \Hom_{*-1}\Con(y_0,A)
	 				\ar[r]&
	 			}
	 			\end{equation}
	 			is a long exact sequence.
	 		\end{theorem}
	 		
	 		\eqref{eq:160107-1253} is called the {\em (long exact) attractor-repeller sequence} of
	 		$(y_0, K, A, R)$. We also say that $\rep{\partial}$ is the connecting homomorphism of
	 		$(y_0, K, A, R)$ respectively of the attractor-repeller sequence associated with $(y_0, K, A, R)$.
	 		
	 		It is an immediate consequence of Theorem \ref{th:150320-1505} below that $\rep{\partial}$
	 		is well-defined. The proof that the morphisms $\rep{i}$ and $\rep{p}$ are well-defined
	 		is omitted. 
	 		
	 		\begin{lemma}
	 			\label{le:141028-1909}
	 			Let $(N_1, N_2, N_3)$ and $(M_1, M_2, M_3)$ be index triples
	 			for $(y_0, K, A, R)$. Then there is an index triple $(L_1, L_2, L_3)$
	 			such that for some $T>0$
	 			\begin{equation}
	 			\label{eq:141028-1946}
	 			(L_1, L_2, L_3) \subset (N_1\cap M_1, N^{-T}_2(N_1)\cap M^{-T}_2(M_1), N^{-T}_3(M_1)\cap M^{-T}_3(M_1)).
	 			\end{equation}
	 		\end{lemma}
	 		
	 		\begin{proof}
	 			By Theorem \ref{th:140128-1546}, there are index pairs $(\tilde L_1, \tilde L_3)$ for $(y_0,K)$
	 			and $(L'_2, L'_3)$ for $(y_0, A)$ which have the required inclusion properties, that is,
	 			for some $T'>0$ it holds that
	 			\begin{align*}
	 			(\tilde L_1, \tilde L_3) &\subset (N_1\cap M_1, N^{-T'}_3(N_1)\cap M^{-T'}_3(M_1)) \\
	 			(L'_2, L'_3) &\subset (\tilde L_1\cap N_2\cap M_2, N^{-T'}_3(N_2)\cap M^{-T'}_3(M_2)).
	 			\end{align*}
	 			
	 			Assume for the moment that there is a constant $T''>0$ such that 
	 			\begin{enumerate}
	 				\item[(*)] $(L'_2 \cup \tilde L^{-T''}_3, \tilde L^{-T''}_3)$ is an index pair for $(y_0, A)$. 
	 			\end{enumerate}
	 			By Lemma \ref{le:140924-1939}, $(\tilde L_1, \tilde L^{-T''}_3)$ is an index pair for $(y_0, K)$, so
	 			by (*)
	 			\begin{equation*}
	 			(L_1, L_2, L_3) := (\tilde L_1, L'_2\cup \tilde L^{-T''}_3, \tilde L^{-T''}_3)
	 			\end{equation*}
	 			is an index triple for $(y_0, K, A, R)$. Furthermore, taking $T=T'+T''$, \eqref{eq:141028-1946} 
	 			is satisfied.
	 			
	 			It is therefore sufficient to prove the assumption above. Firstly, we will show
	 			that $L'_3\subset \tilde L^{-T''}_3 := \tilde L^{-T''}_3(\tilde L_1)$ for $T''$
	 			large enough.
 				Suppose that $(t_n, x_n)\in L'_3 \setminus \tilde L^{-2n}_3(N_1)$
	 				is a sequence. We have 
	 				\begin{equation}
	 				\label{eq:160106-1728}
	 				(t_n,x_n)\chi_{y_0} \left[0,2n\right]\subset N^{-T'}_3(N_2) \subset N^{-T'}_3(N_1)
	 				\end{equation}
	 				for all $n\in\IN$.
	 				$(N_1, N^{-T'}_3(N_1))$ is an index pair for $(y_0,K)$ by virtue of Lemma \ref{le:140924-1939},
	 				so there is an admissible isolating neighborhood $N$ of $K$ such that $r(t_n,x_n)\pi\left[0,2n\right]\subset N$
	 				for all $n\in\IN$.
	 				
	 				We may assume without loss of generality that $r(t_n,x_n)\pi n\to (y,x)\in K$, so $(t_n,x_n)\chi_{y_0} n\in N_1\setminus N^{-T'}_3(N_1)$
	 				provided that $n$ is sufficiently large, in contradiction to \eqref{eq:160106-1728}.
	 				
			 	To prove (*), we need to check the assumptions of an index pair.
 				\begin{enumerate}
 					\item[(IP1)] It is clear that $L_2$ and $L_3$ are closed sets with $L_2\subset L_3$.
 					\item[(IP2)] Let $x\in L_2\setminus L_3$ and $x\chi_{y_0} t\not\in L_2$ for some $t>0$. It follows
 					that $x\chi_{y_0} t\not\in L'_2$, so $x\chi_{y_0} s\in L'_3\subset L_3$ for some $s\in\left[0,t\right]$.
 					\item[(IP3)] Suppose that $x\in L_3$, but $x\chi_{y_0} t\not\in L_3$ for $t>0$. 
 					$(L_1, L_3)$ is an index pair by Lemma \ref{le:140924-1939}, so $x\chi_{y_0} s\in (\IR^+\times X)\setminus L_1 \subset (\IR^+\times X)\setminus L_2$ for 
 					some $s\in\left[0,t\right]$.
 					\item[(IP4)] $(L'_2, L'_3)$ is an index pair for $(y_0, A)$. Hence there is an admissible
 					isolating neighborhood $N\subset \Hull^+(y_0)\times X$ for $A$ with $L_2\setminus L_3\subset L'_2\setminus L'_3\subset r^{-1}(N)$. 
	 				\item[(IP5)] There is a neighborhood $W$ of $A$ in $\Hull^+(y_0)\times X$ such that $r^{-1}(W)\subset L'_2\setminus L'_3$.
					Since $(L_1, L_3)$ is an index pair
	 				for $(y_0, K)$, there is also a neighborhood $W_K$ of $K$ with $r^{-1}(W_K)\subset L_1\setminus L_3$. 
	 				The intersection $W_0 := W\cap W_K$ is a neighborhood of $A$, and $r^{-1}(W_0)\subset L_2\setminus L_3$.
 				\end{enumerate}
	 		\end{proof}
	 		
	 		\begin{theorem}
	 			\label{th:150320-1505}
	 			Let $(N_1, N_2, N_3)$ and $(M_1, M_2, M_3)$ be index triples for $(y_0, K, A, R)$. 
	 			Then the following diagram is commutative. 
	 			\begin{equation*}
	 			\xymatrix{&
	 				\ar[r] &\Hom_*\rep{N_1, N_2} \ar[r]^-\partial \ar[d]&
	 				\Hom_*\rep{N_2, N_3} \ar[r] \ar[d]&\\
	 				&
	 				\ar[r] &\Hom_*\rep{M_1, M_2} \ar[r]^-\partial &
	 				\Hom_*\rep{M_2, M_3}\ar[r]&\\
	 			}
	 			\end{equation*}
	 			
	 			Its rows represent the long exact sequences
	 			associated with the respective index triple, and the vertical arrows denote the
	 			respective inner morphism of the categorial Conley index.
	 		\end{theorem}
	 		
	 		\begin{proof}
	 			Assuming that $(N_1, N_2, N_3)\subset (M_1, M_2, M_3)$, the inner morphisms
	 			are inclusion induced, so the theorem is merely a reformulation of Lemma \ref{le:20141024-1820}.
	 			The general case follows from Lemma \ref{le:141028-1909}. Let the index triple $(L_1, L_2, L_3)$
	 			be given by that lemma. We have
	 			\begin{align*}
	 			(N_1, N_2, N_3) \subset (N_1, N^{-T}_2, N^{-T}_3) \supset (L_1, L_2, L_3)\\
	 			(M_1, M_2, M_3) \subset (M_1, M^{-T}_2, M^{-T}_3) \supset (L_1, L_2, L_3)
	 			\end{align*}
	 			for some $T>0$.
	 			
	 			By Lemma \ref{le:141029-1824}, the triples $(N_1, N^{-T}_2, N^{-T}_3)$
	 			and $(M_1, M^{-T}_2, M^{-T}_3)$ in the middle are index triples. This reduces the general
	 			case to the special case covered by Lemma \ref{le:20141024-1820}.
	 		\end{proof}
	 		
	 	\end{subsection}
	 \end{section}


\providecommand{\bysame}{\leavevmode\hbox to3em{\hrulefill}\thinspace}
\providecommand{\MR}{\relax\ifhmode\unskip\space\fi MR }
\providecommand{\MRhref}[2]{%
	\href{http://www.ams.org/mathscinet-getitem?mr=#1}{#2}
}
\providecommand{\href}[2]{#2}

\end{document}